\documentclass[11pt]{amsart}

\usepackage[a4paper,margin=1in]{geometry}
\usepackage{amsmath,amssymb,amsthm}
\usepackage{hyperref}
\hypersetup{hidelinks}
\usepackage{enumitem}

\newtheorem{theorem}{Theorem}[section]
\newtheorem{lemma}[theorem]{Lemma}
\newtheorem{proposition}[theorem]{Proposition}

\newcommand{\Cstar}{C^{*}}
\newcommand{\N}{\mathbb{N}}
\newcommand{\C}{\mathbb{C}}
\newcommand{\Z}{\mathcal{Z}}
\newcommand{\Cu}{\mathrm{Cu}}
\newcommand{\K}{\mathcal{K}}
\newcommand{\Her}{\mathrm{Her}}
\newcommand{\QT}{\mathrm{QT}}
\newcommand{\U}{\mathcal{U}}
\newcommand{\supp}{\operatorname{supp}}

\makeatletter
\newcommand{\reqnoslocal}{\tagsleft@false\let\veqno\@@eqno}
\makeatother

\newtheorem{innercustomthm}{Theorem}
\newenvironment{customthm}[1]
  {\renewcommand\theinnercustomthm{#1}\innercustomthm}
  {\endinnercustomthm}

\title{Homotopies of Constant Cuntz Classes in Pure $\Cstar$-Algebras}
\author{Chrisil Ouseph}
\address{Department of Mathematics, Purdue University, 150 N University St, West Lafayette, IN 47907, USA}
\email{couseph@purdue.edu}
\author{Andrew S. Toms}
\address{Department of Mathematics, Purdue University, 150 N University St, West Lafayette, IN 47907, USA}
\email{atoms@purdue.edu}

\begin{document}

\begin{abstract}
Let $A$ be a unital, simple, separable, and pure $\Cstar$-algebra. We prove that, for every $a\in A_+$, the set of positive elements in $A$ having the same Cuntz class as $a$ is path-connected. This may be viewed as a generalization to positive elements of the classical homotopy theory of projection classes, and extends a line of results on the topology of constant Cuntz classes that previously required stronger regularity hypotheses to the class of pure $\Cstar$-algebras, which notably contains all selfless algebras.
\end{abstract}
\maketitle

\section{Introduction}

The study of the topology of spaces of projections in operator algebras has a long history. For simple and non-elementary $\Cstar$-algebras with real rank zero and stable rank one, Zhang computed the higher homotopy groups of Murray-von Neumann equivalence classes of projections and related them to $K$-theory \cite{Zhang1991}. Brown proved that in any $\Cstar$-algebra of stable rank one, unitarily equivalent projections are also homotopic through projections \cite{Brown1995}. Together with cancellation of projections in stable rank one, this shows that a fixed Murray-von Neumann equivalence class of projections is path-connected. In the $\Z$-stable setting, Jiang's unpublished work on nonstable $K$-theory established $K$-stability, yielding the same homotopy computations as Zhang \cite{Jiang1997}, and Hua recently gave a new proof using more modern methods from the classification theory of $\Cstar$-algebras \cite{Hua2025}.

Pureness was introduced by Winter as a regularity condition for the Cuntz semigroup, combining comparison and divisibility \cite{Winter2012}, and may be viewed as an analogue of $\Z$-stability at the level of the Cuntz semigroup \cite{AntoinePereraThiel2018}. It therefore fits naturally into the Toms-Winter regularity picture underlying the classification program: in the unital, simple, separable, nuclear, and non-elementary setting, the only direction that remains open in general would show that strict comparison implies pureness, which in turn implies $\Z$-stability. At the same time, pureness makes sense far beyond the nuclear setting. This broader scope has become especially useful recently through the theory of selfless $\Cstar$-algebras introduced in \cite{Robert2025}: every selfless $\Cstar$-algebra is now known to be pure \cite{Gould2026}, and Robert recently showed that any unital, simple, infinite-dimensional, and monotracial $\Cstar$-algebra with strict comparison with respect to its trace is selfless \cite{Robert2026}. Furthermore, selflessness results for broad classes of reduced group algebras, reduced free products, and graph products have produced a rapidly expanding supply of nonnuclear examples \cite{AmrutamGaoKunnawalkamPatchell2025, HayesKunnawalkamRobert2025, FloresKlisseOCobhthaighPagliero2026}.

Cuntz equivalence generalizes Murray-von Neumann equivalence from projections to all positive elements. It is therefore natural to ask which topological features of projection classes persist for arbitrary Cuntz classes. The second author initiated the study of such spaces, proving path-connectedness under substantially stronger structural assumptions including real rank zero and approximate divisibility \cite{Toms2023} and subsequently computing their higher homotopy groups \cite{Toms2024}. We later proved that noncompact constant Cuntz classes are contractible for all unital, simple, and $\Z$-stable algebras \cite{OusephToms2024}, completing the computation of the homotopy groups of all constant Cuntz classes in this setting. The present paper extends the path-connectedness result to the weaker regularity assumption of pureness without any assumptions about the (quasi)tracial state space or other structural properties of the algebra:

\begin{customthm}{I}\label{thm:main}
Let $A$ be a unital, simple, separable, and pure $\Cstar$-algebra. For any $a\in A_+$, the subspace $\{b\in A_+:[b]=[a]\}$ is path-connected.
\end{customthm}

Our proof rests on several delicate constructions. We adapt recent work of the second author and Wan \cite{TomsWan2025} to produce, inside the hereditary subalgebra generated by a given positive element, a Cuntz-equivalent element all of whose quasitracial spectral measures are diffuse away from zero. In its functional calculus, we construct a summable orthogonal decomposition over the complementary intervals of the ternary Cantor set. Slight enlargements of those pieces provide enough room to move them one at a time through null-homotopic unitary conjugations into an arbitrary second representative of the same Cuntz class, while leaving just enough rank for the noncompact tail.

The paper is organized as follows: Section~\ref{sec:preliminaries} recalls essential facts about Cuntz equivalence and dimension functions, then introduces the required generalization of the Toms-Wan result. Section~\ref{sec:decomposition} constructs the orthogonal Cantor decomposition and proves the rank-preserving absorption lemma. Section~\ref{sec:mainproof} iterates the previous lemma and completes the proof of Theorem~\ref{thm:main}.

\subsection*{Acknowledgements}

We would like to thank Hao Wan for allowing us to include the proof of Theorem~\ref{thm:hereditary-prescribed-measures}, which is to appear as part of a follow-up paper to \cite{TomsWan2025}.

\subsection*{AI Statement}

ChatGPT was used for proofreading and reference checking. The mathematical content of the paper is due solely to the authors save for the theorem acknowledged above.

\section{Preliminaries}\label{sec:preliminaries}

We use standard notation and facts about the Cuntz semigroup; see, for example, \cite{AraPereraToms2011,GardellaPerera}. For positive elements $a,b\in A$, we write $a\precsim b$ if there is a sequence $(x_n)$ in $A$ such that $x_nbx_n^*\to a$. We say that $a$ and $b$ are Cuntz equivalent, written $a\sim b$, if $a\precsim b$ and $b\precsim a$. The Cuntz semigroup $\Cu(A)$ is the set of Cuntz equivalence classes $[a]$ of elements $a\in(A\otimes\K)_+$, ordered by Cuntz subequivalence and equipped with addition by orthogonal direct sum. If $x,y\in\Cu(A)$, then $x$ is compactly contained in $y$, written $x\ll y$, if whenever $y\leq\sup_n y_n$ for an increasing sequence $(y_n)$ in $\Cu(A)$, one has $x\leq y_n$ for some $n$. A class $x$ is compact if $x\ll x$.

For a unital $\Cstar$-algebra $A$, write $\QT(A)$ for the set of normalized $2$-quasitraces on $A$ endowed with the topology of pointwise convergence. For $\tau\in\QT(A)$ and $a\in(A\otimes\K)_+$, the associated dimension function is defined as
\[d_\tau(a)=\lim_{n\to\infty}\tau(a^{1/n}).\]
If $a\in A_+$, the restriction of $\tau$ to the abelian subalgebra $\Cstar(1_A,a)$ is a state whose representing probability measure on $\sigma(a)$ will be denoted by $\mu_{\tau,a}$. Thus, $d_\tau(f(a))=\mu_{\tau,a}(\supp_o(f))$, where $\supp_o(f)$ denotes the open support of $f\in C(\sigma(a))_+$. We reserve $\lambda$ for the Lebesgue measure on $[0,1]$ and write $\Her(a)=\overline{aAa}$.

We say that a $\Cstar$-algebra $A$ with $\QT(A)\ne\varnothing$ has strict comparison if, for any $a,b\in(A\otimes\K)_+$, we have $a\precsim b$ whenever $d_\tau(a)<d_\tau(b)$ for every $\tau\in\QT(A)$. The normalized quasitraces are understood with their canonical extensions to matrix algebras and to $A\otimes\K$. Further, we say that a $\Cstar$-algebra $A$ is pure if $\Cu(A)$ is almost unperforated and almost divisible.

The first two lemmas provide us with elementary yet useful homotopies that maintain Cuntz equivalence.

\begin{lemma}\label{lem:spectral-gap-homotopy}
Let $A$ be a unital $\Cstar$-algebra and let $a\in A_+$. Suppose that zero is not an accumulation point in $\sigma(a)$. Then the straight-line homotopy $h(t)=(1-t)a+t\chi_{(0,\infty)}(a)$ for $0\le t\leq1$ is a path of Cuntz-equivalent elements.
\end{lemma}

\begin{proof}
Set $p=\chi_{(0,\infty)}(a)$. The spectral assumption ensures that $p\in \Cstar(1_A,a)$, and that there exists $\delta>0$ with $a\ge\delta p$. For each $t$, the element $h(t)$ belongs to the corner $pAp$ and is bounded below there by $((1-t)\delta+t)p$. Thus $p\precsim h(t)\precsim p$, so $h(t)\sim p\sim a$ for every $t\in[0,1]$.
\end{proof}

\begin{lemma}\label{lem:segment}
Consider a $\Cstar$-algebra $A$ with $a\in A_+$ and $b\in\Her(a)_+$ satisfying $b\sim a$. Then the straight-line homotopy $h(t)=(1-t)a+tb$ for $0\le t\leq1$ is a path of Cuntz-equivalent elements in $\Her(a)_+$.
\end{lemma}

\begin{proof}
The endpoints are Cuntz equivalent by assumption. For $0<t<1$, we observe that $a\sim(1-t)a\le h(t)$ as well as $h(t)\precsim a$, since $h(t)\in\Her(a)$, completing the proof.
\end{proof}

The next proposition combines \cite[Theorem 3.5]{BrownCiuperca2009} and \cite[Proposition 4.3]{GardellaPerera} to give several useful characterizations of compactness (and noncompactness) of Cuntz classes.

\begin{proposition}\label{prop:compact}
Suppose $A$ is a unital $\Cstar$-algebra with stable rank one. For any $a\in A_+$, the following are equivalent:
\begin{enumerate}[label=(\roman*)]
\item The class $[a]$ is compact.
\item There exists some $\varepsilon>0$ such that $\sigma(a)\cap (0,\varepsilon]=\varnothing$.
\item There exists some $\varepsilon>0$ such that $a\sim (a-\varepsilon)_+$.
\item There exists some projection $p\in A$ such that $a\sim p$.
\end{enumerate}
\end{proposition}

Symmetrizing the hypotheses of \cite[Proposition 5.9]{AraPereraToms2011} provides a characterization of Cuntz equivalence for noncompact classes in terms of dimension functions under strict comparison:

\begin{proposition}\label{prop:rank-determines-class}
If $A$ is a unital and simple $\Cstar$-algebra with strict comparison and $a,b\in A_+$ have noncompact Cuntz classes, then
\[a\sim b\quad\Longleftrightarrow\quad d_\tau(a)=d_\tau(b)\ \text{for every }\tau\in\QT(A).\]
\end{proposition}

We show the following generalization of the main theorem of \cite{TomsWan2025}, whose proof was kindly provided to us by Hao Wan and utilizes the notation developed in the same paper.

\begin{theorem}\label{thm:hereditary-prescribed-measures}
Let $A$ be a unital, simple, separable, and pure $\Cstar$-algebra with
$\QT(A)\neq\varnothing$, and let $a\in A_+\setminus\{0\}$. Suppose $h:\QT(A)\longrightarrow \mathcal P([0,1])_+$ is continuous and affine, where $\mathcal P([0,1])_+$ denotes the faithful regular Borel probability measures endowed with the weak topology, and
\[h(\tau)(\{0\})\ge \mu_{\tau,a}(\{0\})\qquad \text{for all } \tau\in\QT(A).\]
Then there exists $b\in \Her(a)_+$ such that $\mu_{\tau,b}=h(\tau)$ for every $\tau\in\QT(A)$ and $\sigma(b)=[0,1]$.
\end{theorem}

\begin{proof}
We can assume that $a$ is a contraction. Put $B=\Her(a)$. Since $\QT(A)\neq\varnothing$, $A$ is not purely infinite, hence has stable rank one by \cite[Theorem~1.4]{Lin2025}. The proof of \cite[Theorem 2.6]{BrownToms2007} now gives the required decomposition of the Cuntz semigroup:
\[\Cu(A)\cong V(A)\sqcup \operatorname{LAff}(\QT(A))_{++}.\]
We adapt \cite[Section~4]{TomsWan2025}. Recall that $\Cu(C([0,1]))\cong\operatorname{Lsc}([0,1],\overline\N)$ and that the function $\gamma_h:\Cu(C([0,1]))\to\Cu(A)$ was defined as follows: a compact element of constant rank $n$ is sent to $n[1_A]$, while a noncompact rank function $f$ is sent to the noncompact element with rank function
\[\tau\longmapsto\int_{[0,1]}f\,dh(\tau).\]
The latter function is lower semicontinuous and affine by the argument in
\cite[Lemma~4.1]{TomsWan2025}, and is strictly positive for $f\neq0$ because every $h(\tau)$ is faithful. Thus $\gamma_h$ is well defined. The proofs of \cite[Proposition~4.2]{TomsWan2025} and the auxiliary results it uses require only stable rank one, strict comparison, the decomposition above, its mixed order, and faithfulness of the measures, and therefore apply without change. Consequently, $\gamma_h$ is a Cu-morphism.

Next, we force the lift into $B$. Let $j:C_0((0,1])\hookrightarrow C([0,1])$ and $\iota:B\hookrightarrow A$ be the usual inclusions. As $B$ is full hereditary, $\Cu(\iota)$ is an isomorphism by \cite{CowardElliottIvanescu2008}, so we can define the $\Cu$-morphism $\beta=\Cu(\iota)^{-1}\circ\gamma_h\circ\Cu(j)$. If $t(s)=s$ is the canonical generator of $C_0((0,1])$, then $\operatorname{rank}(j(t))=\chi_{(0,1]}$ and for any $\tau\in\QT(A)$,
\[d_\tau\!\left(\Cu(\iota)(\beta([t]))\right)=\int_{[0,1]}\chi_{(0,1]}\,dh(\tau)=1-h(\tau)(\{0\})\leq1-\mu_{\tau,a}(\{0\})=d_\tau(a).\]
The class $\Cu(\iota)(\beta([t]))$ is noncompact since $\gamma_h$ preserves noncompactness by construction. Whether $[a]$ is compact or not, the order structure on the decomposition of $\Cu(A)$ gives $\Cu(\iota)(\beta([t]))\le[a]$, and therefore $\beta([t])\le[a]$ in $\Cu(B)$. Since stable rank one passes to hereditary subalgebras and $a$ is strictly positive in $B$, Theorem~2.15 of \cite{ElliottIm2024}, first proved in \cite{CiupercaElliott2008}, yields a $*$-homomorphism
$\Phi:C_0((0,1])\to B$ with $\Cu(\Phi)=\beta$. Put $b=\Phi(t)$.

The same argument as in the proof of \cite[Theorem~1.1]{TomsWan2025} gives that for every open $O\subset(0,1]$, $\mu_{\tau,b}(O)=h(\tau)(O)$. The equality of probability measures $\mu_{\tau,b}=h(\tau)$ follows. Since $\QT(A)\neq\varnothing$ and each $h(\tau)$ is faithful on $[0,1]$, no nonempty open subset of $[0,1]$ can be disjoint from $\sigma(b)$; therefore $\sigma(b)=[0,1]$.
\end{proof}

\section{Spectral Measures, Cantor Decomposition, and Rank Absorption}\label{sec:decomposition}

The first proposition of the section picks out a noncompact representative with full spectrum and quasitracial spectral measures that are diffuse away from zero. This will give us substantial flexibility in functional calculus while maintaining Cuntz equivalence in the sequel.

\begin{proposition}\label{prop:Lebesgue-spectrum}
Consider a unital, simple, separable, and pure $\Cstar$-algebra $A$ with $\QT(A)\neq\varnothing$, as well as $a\in A_+$ with noncompact Cuntz class. Then there exists $b\in\Her(a)_+$ such that $\sigma(b)=[0,1]$, $b\sim a$, and for any $\tau\in\QT(A)$, the restriction of $\mu_{\tau,b}$ to $(0,1]$ is absolutely continuous with respect to $\lambda|_{(0,1]}$.
\end{proposition}

\begin{proof}
We assume without loss of generality that $a$ is a contraction. Let $\gamma_n,\gamma:\QT(A)\to[0,1]$ be given by $\gamma_n(\tau)=\tau(a^{1/n})$ for $n\ge 1$ and $\gamma(\tau)=d_\tau(a)$. Since $A$ is simple and $a\neq0$, every normalized quasitrace is faithful, and hence $\gamma(\tau)>0$. Also, the continuous affine functions $0=\gamma_0\le\gamma_1\le\gamma_2\le\cdots\le\gamma$ converge pointwise to $\gamma$. Set $\Delta_n=\gamma_n-\gamma_{n-1}$ for all $n\ge 1$. Each $\Delta_n$ is also continuous, affine, and nonnegative with
$\sum_{n}\Delta_n(\tau)=\gamma(\tau)$ for each $\tau\in\QT(A)$. For all $n\ge 1$, let $\varepsilon_n=2^{-n}$ and $\nu_n\in \mathcal{P}([0,1])_+$ be defined as follows:
\[\nu_n=(1-\varepsilon_n)\frac{\lambda|_{[0,\,\varepsilon_n]}}{\varepsilon_n} +\varepsilon_n\lambda.\]
Each $\nu_n$ is atomless and absolutely continuous with respect to $\lambda$. Also $\nu_n$ converges weakly to the Dirac measure centered at zero, $\delta_0$. Define an affine map $h:\QT(A)\to \mathcal P([0,1])$ given by
\[h(\tau)=(1-\gamma(\tau))\delta_0+\sum_{n=1}^\infty\Delta_n(\tau)\nu_n.\]
For each $\tau$, some $\Delta_n(\tau)$ is positive, and every $\nu_n$ is
faithful; hence $h(\tau)$ is faithful.

We claim that $h$ is continuous. Indeed, if $f\in C([0,1])$ and $c_n=\int f\,d\nu_n-f(0)$, then we observe that $c_n\to0$ and
\[\int f\,dh(\tau)=\left(1-\sum_{n=1}^\infty\Delta_n(\tau)\right)f(0)+\sum_{n=1}^\infty\Delta_n(\tau)\int f\,d\nu_n=f(0)+\sum_{n=1}^\infty\Delta_n(\tau)c_n.\]
The partial sums on the right are continuous affine functions of $\tau$, while $\sum_n\Delta_n(\tau)\leq1$ gives
\[\sup_{\tau\in\QT(A)}\left|\sum_{n>N}\Delta_n(\tau)c_n\right|\le \sup_{n>N}|c_n|.\]
Thus the series converges uniformly in $\tau$, proving the claim. Also,
\[h(\tau)(\{0\})=1-\gamma(\tau)=1-d_\tau(a)=\mu_{\tau,a}(\{0\}).\]
Theorem~\ref{thm:hereditary-prescribed-measures} now gives some $b\in\Her(a)_+$ such that $\sigma(b)=[0,1]$ and $\mu_{\tau,b}=h(\tau)$ for any $\tau\in\QT(A)$. From the definition of $h$, we observe that
\[\mu_{\tau,b}|_{(0,1]}=\sum_{n=1}^\infty\Delta_n(\tau)\nu_n|_{(0,1]}\ll\lambda|_{(0,1]}.\]
Since $b$ has full spectrum and $A$ has stable rank one by \cite[Theorem~1.4]{Lin2025}, Proposition~\ref{prop:compact} shows that $[b]$ is noncompact. Finally, Proposition~\ref{prop:rank-determines-class} and the computation below give $b\sim a$:
\[d_\tau(b)=\mu_{\tau,b}((0,1])=h(\tau)((0,1])=1-h(\tau)(\{0\})=d_\tau(a).\qedhere\]
\end{proof}

The next proposition allows us to replace an arbitrary noncompact representative by an orthogonal decomposition with controlled enlargements. These enlargements remain compatible with the decomposition and will be needed for unitary conjugations in the iterative absorption.

\begin{proposition}\label{prop:decomposition}
Suppose $A$ and $a\in A_+$ satisfy the hypotheses of Proposition~\ref{prop:Lebesgue-spectrum}. Then there exist sequences $(s_i)_{i\in\N}$ and $(\bar s_i)_{i\in\N}$ in $\Her(a)_+\setminus\{0\}$ along with $\varepsilon_i>0$ satisfying the following:
\begin{enumerate}[label=(\roman*)]
\item For every $i\in\N$, $s_i=(\bar s_i-3\varepsilon_i)_+$.
\item $\sum_i\|s_i\|< \infty$, and if $i<j$, then $s_is_j=0=s_i\bar s_j$.
\item For all $i\in\N$, $\alpha\in[0,\|\bar s_i\|)$, and $\tau\in\QT(A)$, the restriction of $\mu_{\tau,(\bar s_i-\alpha)_+}$ to its positive spectrum is atomless.
\item For each $n$, the series $r_n:=\sum_{i=n}^\infty\bar s_i$ converges in $\Her(a)_+$ with noncompact Cuntz class bounded above by $[a]$.
\item There is a path of Cuntz-equivalent elements in $\Her(a)_+$ from $a$ to $\sum_{i} s_i$.
\end{enumerate}
\end{proposition}

\begin{proof}
Apply Proposition~\ref{prop:Lebesgue-spectrum} to get $b\in\Her(a)_+$ satisfying its conclusions. Via functional calculus we identify $\Cstar(b)$ with $C_0((0,1])$. Let $\mathcal C$ denote the ternary Cantor set. At step $i\ge 1$ in the construction of $\mathcal{C}$, set $\rho_i={1}/(2\cdot3^i)$ and write the union of the removed open intervals as
\[O_i=\bigsqcup_{j=1}^{2^{i-1}}(t_j^i-\rho_i,t_j^i+\rho_i).\]
Enumerate the centers in increasing order: $3\rho_i=t_1^i<t_2^i<\cdots<t_{2^{i-1}}^i=1-3\rho_i$. Also, set $\eta_i={3\rho_i}/{2}$ and define piecewise-linear maps $f_i,\bar f_i:[0,1]\to [0,1]$ by
\[
f_i(t)=\sum_{j=1}^{2^{i-1}}(\rho_i-|t-t_j^i|)_+
\qquad \text{and } \qquad \bar f_i(t)=\sum_{j=1}^{2^{i-1}}(\eta_i-|t-t_j^i|)_+.
\]
Finally, put $\varepsilon_i={\rho_i}/{6}$, $s_i=f_i(b)$, and $\bar s_i=\bar f_i(b)$. Note that $O_i$ is the open support of $f_i$, and both $f_i$ and $\bar f_i$ are clearly nonzero. We now verify the conclusions in order:

\begin{enumerate}[label=\textup{(\roman*)},leftmargin=*]
\item We first check that the summands of $\bar f_i$ have disjoint supports. Every endpoint of a component of $O_i$ is an integral multiple of $3^{-i}=2\rho_i$. Hence two distinct components of $O_i$ are separated by at least $2\rho_i$. Passing from a component of $O_i$ to the corresponding support component of $\bar f_i$ enlarges it on each side by $\eta_i-\rho_i={\rho_i}/{2}=3\varepsilon_i$. Thus distinct support components of $\bar f_i$ remain separated by at least $2\rho_i-6\varepsilon_i=\rho_i>0$. At each point, therefore, at most one summand in the definition of $\bar f_i$ is nonzero. Thus, we have $f_i=(\bar f_i-3\varepsilon_i)_+$ and $s_i=(\bar s_i-3\varepsilon_i)_+$.

\item The $s_i$ are mutually orthogonal as the $O_i$ are pairwise disjoint. Also, $\sum_i\|s_i\|=\sum_i\rho_i<1$. Fix $i<j$. Every endpoint of a component of $O_i$ or $O_j$ is an integral multiple of $3^{-j}$. Since any intervals deleted at stage $i$ and stage $j$ have disjoint closures, the distance between any such pair of components is also a positive integral multiple of $3^{-j}=2\rho_j$. The support of $\bar f_j$ is obtained from $O_j$ by enlarging each component on each side by only $3\varepsilon_j={\rho_j}/{2}$ as discussed in the previous part. It follows that the support of $f_i$ remains disjoint from the support of $\bar f_j$, and therefore $s_i\bar s_j=0$.

\item Fix $i\in\N$, $\tau\in\QT(A)$, $y>0$, and $0\le\alpha<\eta_i$. Put $g=(\bar f_i-\alpha)_+$. The level set $g^{-1}(\{y\})\subset(0,1]$ is finite, since $g$ is a finite sum of disjoint piecewise-linear tents. The spectral measure of $g(b)=(\bar s_i-\alpha)_+$ is the push-forward of $\mu_{\tau,b}$ under $g$, the latter having an atomless restriction to the positive spectrum by Proposition~\ref{prop:Lebesgue-spectrum}. Therefore,
\[\mu_{\tau,(\bar s_i-\alpha)_+}(\{y\})=\mu_{\tau,b}(g^{-1}(\{y\}))=0.\]

\item
Since $\|\bar s_i\|=\eta_i$ is summable, each series $r_n=\sum_{i=n}^\infty\bar s_i$ converges in norm to an element of $\Cstar(b)\subseteq\Her(a)$; thus $r_n\precsim a$. Since the scalar function corresponding to $r_n$ vanishes at $t=1$ and is positive at, say, $t=t_1^n$, the intermediate value theorem ensures that $\sigma(r_n)$ contains an interval starting at zero. Thus Proposition~\ref{prop:compact} shows that $[r_n]$ is noncompact.

\item
Set $s=\sum_{i} s_i$. The open support of the scalar function defining $s$ is $[0,1]\setminus\mathcal C$. Since the Cantor set has Lebesgue measure zero, Proposition~\ref{prop:Lebesgue-spectrum} gives $\mu_{\tau,b}(\mathcal C\setminus\{0\})=0$. So, we have for every $\tau\in\QT(A)$,
\[d_\tau(s)=\mu_{\tau,b}([0,1]\setminus\mathcal C)=\mu_{\tau,b}((0,1])=d_\tau(b)=d_\tau(a).\]
The intermediate value theorem and Proposition~\ref{prop:compact} again show that $[s]$ is noncompact. Proposition~\ref{prop:rank-determines-class} gives $s\sim a$, and Lemma~\ref{lem:segment} supplies the required path. \qedhere
\end{enumerate}
\end{proof}

We finally turn to the absorption mechanism which moves an orthogonal summand into a prescribed representative while preserving the total rank available for the remaining tail. The atomlessness condition is precisely what makes the small loss in the rank estimate tend to zero.

\begin{lemma}\label{lem:absorption}
Let $A$ be a unital and simple $\Cstar$-algebra with strict comparison and stable rank one. Also, let $b,\bar s,r\in A_+$, $x,y\in\Her(b)_+$, and $\varepsilon>0$ satisfy the following conditions:
\begin{enumerate}[label=(\roman*)]
\item For every $\tau\in\QT(A)$ and every $0\le\rho<\|\bar s\|$, the restriction of $\mu_{\tau,(\bar s-\rho)_+}$ to the positive spectrum of $(\bar s-\rho)_+$ is atomless.
\item The class $[r]$ is noncompact.
\item $x\perp y$, $x\perp\bar s$, $x\perp r$, and $s\perp r$, where $s=(\bar s-3\varepsilon)_+\ne0$.
\item For every $\tau\in\QT(A)$, $d_\tau(x)+d_\tau(y)=d_\tau(b)$.
\item $\bar s+r\precsim y$.
\end{enumerate}
Then there exist $x',y'\in\Her(b)_+$ and a norm-continuous path $u:[0,1]\to\U(A)$ such that
\begin{enumerate}[label=(\alph*)]
\item $u_0=1_A$ and conjugation by $u_t$ fixes $\Her(x)$ pointwise for all $t$.
\item $x\perp u_1su_1^*$ and $x+u_1su_1^*=x'$.
\item $x'\perp y'$.
\item For every $\tau\in\QT(A)$, $d_\tau(x')+d_\tau(y')=d_\tau(b)$.
\item $r\precsim y'$.
\end{enumerate}
\end{lemma}

\begin{proof}
Let $B=\Her(\bar s+r+y)$. This hereditary subalgebra is nonzero and simple and has stable rank one. Since $x$ is orthogonal to $\bar s$, $r$, and $y$, we have $xB=Bx=0$. Put $B^\dagger=B+\C1_A\subseteq A$. $B^\dagger$ also has stable rank one, and $K_1(B^\dagger)$ identifies canonically with $K_1(B)$.

Since $\bar s\precsim\bar s+r\precsim y$ and hereditary inclusions induce order-embedding Cuntz morphisms \cite{CowardElliottIvanescu2008}, $\bar s\precsim y$ holds inside $B$. Proposition~2.17 of \cite{AraPereraToms2011} gives $\delta>0$ such that $(\bar s-\varepsilon)_+\precsim(y-\delta)_+$. Set $D=\Her_B((y-\delta)_+)\subseteq \Her(b)$ and $D^\dagger=D+\C1_A\subseteq B^\dagger$. Another application of the proposition gives a unitary $w\in\U(B^\dagger)$ such that $w(\bar s-2\varepsilon)_+w^*\in D$. This element is nonzero, so $D$ is a full hereditary subalgebra of $B$ and also has stable rank one. The inclusion $D\hookrightarrow B$ now induces an isomorphism on $K_1$ (see \cite{Brown1977}). Under the identifications $\U(D^\dagger)/\U_0(D^\dagger)\cong K_1(D)$ and $\U(B^\dagger)/\U_0(B^\dagger)\cong K_1(B)$, the map on $K_1$ is induced by the inclusion $D^\dagger\hookrightarrow B^\dagger$. Therefore, we can choose $v\in\U(D^\dagger)$ whose $K_1$-class maps to $-[w]_1$. We then have $[vw]_1=0$ in $K_1(B^\dagger)$ and hence $vw\in\U_0(B^\dagger)$, the path component of $1_A$ in $\U(B^\dagger)$. Choose a norm-continuous path $(u_t)_{t\in[0,1]}$ in $\U(B^\dagger)$ such that $u_0=1_A$ and $u_1=vw$.

Set $c=u_1(\bar s-2\varepsilon)_+u_1^*$. Because $v\in D^\dagger$ and $w(\bar s-2\varepsilon)_+w^*\in D$, one has $c\in D$. Given any $\gamma>0$, define $g_\gamma(t)=\min\{1,t/\gamma\}$ and $e=g_\delta(y)$. Since $e$ is equal to $1$ on the support of $(y-\delta)_+$, it acts as a unit on $D$. In particular, $ec=ce=c$. Define $x',y'\in \Her(b)_+$ as $x'=x+(c-\varepsilon)_+$ and $y'=e(1-g_\varepsilon(c))e$.

We now verify each of the conclusions:

\begin{enumerate}[label=\textup{(\alph*)},leftmargin=*]
\item
If $B=A$, then $xB=0$ forces $x=0$. Otherwise, write $u_t=u'+\eta1_A$ for some $u'\in B$ and $|\eta|=1$; if $z\in\Her(x)$, then $u'z=zu'=0$, so $u_tzu_t^*=z$.

\item
The first part follows from $xB=0$, the second from $(c-\varepsilon)_+=u_1(\bar s-3\varepsilon)_+u_1^*=u_1su_1^*$.

\item
We have $xe=0$, and $(c-\varepsilon)_+(1-g_\varepsilon(c))=0$. The orthogonality of $x'$ and $y'$ now follows from the fact that $e$ commutes with $c$, hence with $(c-\varepsilon)_+$.

\item
Define maps $F(t)=(t-\varepsilon)_++1-g_\varepsilon(t)$ and $h_n(t)=\min\{1,n^2\left({t}/{\varepsilon}-1\right)^2\}$ for $n\in\N$ and $t\in[0,\|c\|]$. The function $F$ vanishes only at $t=\varepsilon$. That $e$ acts as a unit on $D$ implies that $eF(c)e=(c-\varepsilon)_++y'$.

Fix any $\tau\in\QT(A)$. We claim that $d_\tau(e)= d_\tau(eF(c)e)$. Indeed, for each $n$, the quotient $h_n/F$, defined away from $\varepsilon$, extends continuously across $\varepsilon$ by taking the value $0$ there. Thus there is $k_n(c)\in \Cstar(c,1_A)_+$ such that $h_n(c)=k_n(c)F(c)$, and hence
\[eh_n(c)e=k_n(c)^{1/2}eF(c)e\,k_n(c)^{1/2}\precsim eF(c)e.\]
Since $e\sim e^2=eh_n(c)e+e(1-h_n(c))e$, subadditivity of dimension functions gives
\begingroup
\reqnoslocal
\renewcommand{\theHequation}{ineqstar}
\begin{equation}
d_\tau(e)
\le d_\tau(eh_n(c)e)+d_\tau(e(1-h_n(c))e)
\le d_\tau(eF(c)e)+d_\tau(e(1-h_n(c))e).
\tag{*}\label{eq:inequality}
\end{equation}
\endgroup
Since $e$ is a contraction, we also get that
\[e(1-h_n(c))e\sim(1-h_n(c))^{1/2}e^2(1-h_n(c))^{1/2}\leq1-h_n(c).\]
The element $c$ is unitarily equivalent to $(\bar s-2\varepsilon)_+$, so
hypothesis (i) implies that $\mu_{\tau,c}$ also has no atoms in its positive
spectrum. Moreover, the open support of $1-h_n$ is $I_n=(\varepsilon(1-1/n), \varepsilon(1+1/n))$. Note that $I_n$ decreases to $\{\varepsilon\}$ and never contains zero. Thus
\[\lim_{n\to\infty}d_\tau(e(1-h_n(c))e)\le\lim_{n\to\infty}d_\tau(1-h_n(c))=\lim_{n\to\infty}\mu_{\tau,c}(I_n)=\mu_{\tau,c}(\{\varepsilon\})=0.\]
Therefore, letting $n\to\infty$ in \eqref{eq:inequality} gives $d_\tau(e)\le d_\tau(eF(c)e)$. The reverse inequality holds because $eF(c)e\in\Her(e)$, hence $eF(c)e\precsim e$, proving the claim. Hence
\[d_\tau((c-\varepsilon)_+)+d_\tau(y')=d_\tau((c-\varepsilon)_++y')=d_\tau(e)=d_\tau(y),\]
where the last equality follows because $e=g_\delta(y)$ has the same open support as $y$. Adding $d_\tau(x)$ to both sides of the equation above yields the desired equality.

\item
We get $r\precsim y'$ from \cite[Proposition~5.9]{AraPereraToms2011} and the inequality below for any $\tau\in\QT(A)$:
\[d_\tau(r)=d_\tau(s+r)-d_\tau(s)\le d_\tau(\bar s+r)-d_\tau(u_1su_1^*)\le d_\tau(y)-d_\tau((c-\varepsilon)_+)=d_\tau(y').\qedhere\]
\end{enumerate}
\end{proof}

\section{Proof of the main theorem}\label{sec:mainproof}

We are now ready to prove our main theorem, restated here:
\begin{customthm}{I}
Let $A$ be a unital, simple, separable, and pure $\Cstar$-algebra. For any $a\in A_+$, the subspace $\{b\in A_+:b\sim a\}$ is path-connected.    
\end{customthm}

\begin{proof}
Fix $a,b\in A_+$ with $a\sim b$. If $a=0$, then there is nothing to prove. Suppose first that $A$ is purely infinite. In this case, any pair of nonzero positive elements is Cuntz equivalent. The straight-line homotopy $f(t)=(1-t)a+tb$ gives a nonzero positive element for each $t\in[0,1]$, hence is a path from $a$ to $b$ of Cuntz-equivalent elements.

By the dichotomy in \cite[Theorem~1.4]{Lin2025}, we assume henceforth that $A$ has stable rank one. Then stable finiteness and the Blackadar-Handelman theorem \cite{BlackadarHandelman1982} give $\QT(A)\neq\varnothing$. Moreover, pureness also implies strict comparison. Suppose first that the common Cuntz class is compact. Proposition~\ref{prop:compact} and Lemma~\ref{lem:spectral-gap-homotopy} give constant-class paths from $a$ and $b$ to their support projections $p$ and $q$, respectively. Since $p\sim q$, the projections are Murray-von Neumann equivalent. Stable rank one ensures cancellation of projections, so $1-p$ and $1-q$ are also Murray-von Neumann equivalent; consequently $p$ and $q$ are unitarily equivalent. They can be joined by a path of projections by \cite[Theorem~1]{Brown1995}. Concatenating these three paths joins $a$ to $b$ inside their common Cuntz class.

Assume from now on that $[a]=[b]$ is noncompact. We may also assume without loss of generality that $a=\sum_{i=1}^\infty s_i$, with $s_i$, $\bar s_i$, $\varepsilon_i$, and $r_n$ satisfying all conclusions of Proposition~\ref{prop:decomposition}. Pick a sequence $(t_k)_{k=0}^\infty$ such that $t_0=0$ and $t_k\nearrow1$. We first inductively construct a continuous path $v:[0,1)\longrightarrow\U(A)$, and, for every $k\geq1$, orthogonal elements $x_k,y_k\in\Her(b)_+$ such that
\begin{align}
&d_\tau(x_k)+d_\tau(y_k)=d_\tau(b) \text{ for all } \tau\in\QT(A),\label{eq:ind-rank}\\
&r_{k+1}\precsim y_k,\label{eq:ind-tail}\\
&x_k=v_{t_k}\left(\sum_{i=1}^k s_i\right)v_{t_k}^*=\sum_{i=1}^kv_{t_i}s_iv_{t_i}^*,\text{ and }\label{eq:ind-partial}\\
&v_t\left(\sum_{i=1}^j s_i\right)v_t^*=x_j \text{ for all } t\in[t_j,t_k] \text{ where } 1\le j\le k.\label{eq:fixing-invariant}
\end{align}
For the first step, apply Lemma~\ref{lem:absorption} with $x=0$, $y=b$, $\bar s=\bar s_1$, $r=r_2$, and $\varepsilon=\varepsilon_1$. All hypotheses hold trivially. Reparameterizing the resulting unitary path to $[0,t_1]$, we obtain $v_t$, $x_1$, and $y_1$ satisfying \eqref{eq:ind-rank}-\eqref{eq:ind-partial}; \eqref{eq:fixing-invariant} can be trivially verified for $t=t_1$.

Suppose that the construction has been completed through $t=t_k$. Apply Lemma~\ref{lem:absorption} with $x=x_k$, $y=y_k$, $\bar s=v_{t_k}\bar s_{k+1}v_{t_k}^*$, $r=v_{t_k}r_{k+2}v_{t_k}^*$, and $\varepsilon=\varepsilon_{k+1}$. The spectral-measure and noncompactness hypotheses are invariant under unitary conjugacy, and the rank condition is exactly \eqref{eq:ind-rank}. We get $x_k\perp y_k$ by construction, whereas the other orthogonalities follow from \eqref{eq:ind-partial} and Proposition~\ref{prop:decomposition}(ii). Finally, $\bar s+r=v_{t_k}r_{k+1}v_{t_k}^*\precsim y_k$ by \eqref{eq:ind-tail}.

Reparametrize the path supplied by Lemma~\ref{lem:absorption} as $(u_t)_{t\in[t_k,t_{k+1}]}$ and extend $v$ to $[t_k,t_{k+1}]$ as $v_t=u_tv_{t_k}$. Let $x_{k+1}$ and $y_{k+1}$ be the elements supplied by the lemma. The conclusions of the lemma give \eqref{eq:ind-rank} and \eqref{eq:ind-partial} with $k+1$ in place of $k$. \eqref{eq:ind-tail} holds for the next stage since $r_{k+2}\sim v_{t_k}r_{k+2}v_{t_k}^*\precsim y_{k+1}$. Finally, conjugating by $u_t$ fixing $\Her(x_k)$ pointwise for all $t\in [t_k,t_{k+1}]$ ensures that \eqref{eq:fixing-invariant} persists, completing the induction.

The series $z=\sum_i v_{t_i}s_iv_{t_i}^*$ converges as $\sum_i\|s_i\|<\infty$. Also, $z=\lim_{k\to\infty}x_k\in\Her(b)$. Set
\[
h(t)=
\begin{cases}
v_tav_t^*,&0\le t<1,\\
z,&t=1.
\end{cases}
\]
For $t<1$, the element $h(t)$ is unitarily equivalent to $a$. If $t\ge t_k$, then \eqref{eq:fixing-invariant} gives
\[\|z-v_tav_t^*\|\le\left\|\sum_{i>k}v_{t_i}s_iv_{t_i}^*\right\|+\left\|v_t\left(\sum_{i>k}s_i\right)v_t^*\right\|\leq2\sum_{i>k}\|s_i\|,\]
the upper bound going to zero as $k\to \infty$. Hence $h$ is continuous on $[0,1]$.

Next, we verify that $z\sim b$. Put $w_i=v_{t_i}s_iv_{t_i}^*$. These are mutually orthogonal by construction. The sequences $x_k$ and $\sum_{i=1}^ks_i$ are increasing in operator order and converge respectively to $z$ and $a$. Thus, standard supremum properties from \cite{CowardElliottIvanescu2008} show that in $\Cu(A)$,
\[[z]=\sup_k[x_k]=\sup_k\left[\sum_{i=1}^kw_i\right]=\sup_k\sum_{i=1}^k\left[w_i\right]=\sup_k\sum_{i=1}^k\left[s_i\right]=\sup_k\left[\sum_{i=1}^ks_i\right]=[a]=[b].\]
Therefore, we get a path from $z$ to $b$ of Cuntz-equivalent elements from Lemma~\ref{lem:segment}. Concatenation with $h$ gives a path from $a$ to $b$ entirely inside their common Cuntz class.
\end{proof}

\end{document}